\documentclass[10pt,letterpaper,reqno]{amsart}
\usepackage[T1]{fontenc}
\usepackage{amsmath,amssymb,amsthm,mathtools}
\usepackage{xcolor}
\usepackage{hyperref}
\usepackage{doi}
\hypersetup{colorlinks=true,linkcolor=blue,citecolor=blue,urlcolor=blue}
\usepackage{enumitem}
\usepackage{mathrsfs}

\newtheorem{thm}{Theorem}[section]
\newtheorem{lem}[thm]{Lemma}
\newtheorem{prop}[thm]{Proposition}
\newtheorem{cor}[thm]{Corollary}

\newtheorem{ques}[thm]{Question}

\theoremstyle{definition}

\newtheorem{rem}[thm]{Remark}
\newtheorem{defin}[thm]{Definition}
\numberwithin{equation}{section}

\newcommand{\M}{\mathcal M}
\newcommand{\N}{\mathcal N}
\newcommand{\E}{\mathcal E}

\newcommand{\id}{\operatorname{Id}}
\newcommand{\supp}{\operatorname{supp}}
\DeclareMathOperator{\Tr}{Tr}
\DeclareMathOperator{\Ree}{Re}
\DeclareMathOperator{\sgn}{sgn}
\newcommand{\pair}[2]{\left\langle #1,#2\right\rangle}

\begin{document}
	
\title[Sharp tangent inequalities in noncommutative $L_p$-spaces]{Sharp tangent inequalities in noncommutative $L_p$-spaces}
	
	\author[Z.~Xu]{Zongben Xu}
	
	\address{School of Mathematics and Statistics, Xi'an Jiaotong University,
		Xi'an 710049, P.~R.~China}
	\email{zbxu@mail.xjtu.edu.cn}

	\author[T.~Zhang]{Teng Zhang}
	
	\address{School of Mathematics and Statistics, Xi'an Jiaotong University, Xi'an 710049, P. R. China}
	\email{teng.zhang@stu.xjtu.edu.cn}
	
	
	
\thanks{Teng Zhang is supported by the China Scholarship Council, the Young Elite
	Scientists Sponsorship Program for PhD Students of the China Association for
	Science and Technology, and the Fundamental Research Funds for the Central
	Universities at Xi'an Jiaotong University (Grant No.~xzy022024045).}

\begin{abstract}
	Let \(\M\) be a von Neumann algebra equipped with a faithful normal
	semifinite weight, and let \(L_p(\M)\) denote the associated Haagerup
	noncommutative \(L_p\)-space. For \(1<p<\infty\), let
	\(q=p/(p-1)\), and 
	for \(z\in L_q(\M)\) and \(y\in L_p(\M)\), denote by
	\(\langle z,y\rangle=\Ree\Tr(z^*y)\) the canonical real duality pairing
	between \(L_q(\M)\) and \(L_p(\M)\).
	 Define the normalized duality map by
	\(J_p(x)=\|x\|_p^{\,2-p}u|x|^{p-1}\) for \(x=u|x|\neq0\), and set
	\(J_p(0)=0\). We establish the sharp tangent
	inequality
	\[
	\|x+y\|_p^2
	\ge
	\|x\|_p^2
	+2\langle J_p(x),y\rangle
	+(p-1)\|y\|_p^2,
	\qquad
	x,y\in L_p(\M),\quad 1<p\le2.
	\]
The reverse inequality holds for \(2\le p<\infty\), and the constant
\(p-1\) is optimal. Our inequality extends Xu's tangent inequality \cite{Xu89} from
classical \(L_p\)-spaces to general noncommutative \(L_p\)-spaces and
provides the sharp tangent formulation of the Ricard--Xu convexity
inequality \cite{RX16}. When specialized to Schatten \(p\)-classes, it
also yields the corresponding tangent formulation of the celebrated
Ball--Carlen--Lieb convexity inequality \cite{BCL94}. Under a suitable
positivity assumption, we further establish a sharp complement to Xu's
tangent inequality for Schatten \(p\)-classes in the range \(0<p<1\).

\end{abstract}

	\maketitle
	
%	\tableofcontents
	
	\section{Introduction}

	For a complex inner-product space
	\((V,\langle\cdot,\cdot\rangle)\), the norm induced by the inner product
	satisfies
	\begin{equation*}
		\|x+y\|^{2}
		=\|x\|^{2}+2\Ree\langle x,y\rangle+\|y\|^{2},
		\qquad x,y\in V.
	\end{equation*}
	By the Jordan--von Neumann theorem \cite{JN35}, a complex normed space
	\((X,\|\cdot\|)\) is an inner-product space, in the sense that its norm is
	induced by an inner product, if and only if its norm satisfies the
	parallelogram identity
	\begin{equation*}
		\|x+y\|^{2}+\|x-y\|^{2}
		=2\bigl(\|x\|^{2}+\|y\|^{2}\bigr),
		\qquad x,y\in X.
	\end{equation*}
	Thus, the parallelogram identity provides a fundamental characterization
	of inner-product spaces among normed spaces.

	The motivation in this paper comes from the classical program of
	replacing the inner product by a duality mapping in Banach spaces.
	This viewpoint originated in the early work of Beurling and Livingston
	on duality mappings~\cite{BL62} and was subsequently developed in connection with
	monotone nonlinear operators by Browder~\cite{Bro65}, Asplund~\cite{Asp67}, Kato~\cite{Kat67}, and others; see also Cioranescu's monograph
	\cite[Chapters~I--II]{Cio90} for a systematic account.

	Let \(X\) be a uniformly smooth real Banach space. For every
	\(x\in X\), there exists a unique functional \(J(x)\in X^*\) such that
	\[
	\langle J(x),x\rangle=\|x\|^2,
	\qquad
	\|J(x)\|_{X^*}=\|x\|;
	\]
	see \cite[p.~27, Corollary~4.5]{Cio90}.
	Indeed, the existence of such a functional follows from the
	Hahn--Banach theorem, while its uniqueness follows from the smoothness
	of \(X\). The resulting single-valued map
	$
	J:X\longrightarrow X^*
	$
	is called the \emph{normalized duality map}, with \(J(0)=0\).
	Here \(\langle\cdot,\cdot\rangle\) denotes the canonical duality
	pairing between \(X^*\) and \(X\). Whenever no confusion can arise, we
	use the same bracket notation for duality pairings and inner products.
	The map \(J\) may be regarded as the nonlinear Banach-space analogue
	of the Riesz map in a Hilbert space. Accordingly, the duality term
	$
	\langle J(x),y\rangle
	$
	plays the role of the inner-product term \(\langle x,y\rangle\) in
	Hilbert-space identities.

	Define
	$
	\Psi:X\to\mathbb R,
	$ with $
	\Psi(x):=1/2\|x\|^2.
	$
	Then \(\Psi\) is Fr\'echet differentiable on \(X\), and
	$
	D\Psi(x)=J(x)\in X^*.
	$
	Consequently, for every \(x,h\in X\),
	$
	D\Psi(x)[h]
	=
	\langle J(x),h\rangle.
	$
	Thus,
	\[
	\Psi(x+h)
	=
	\Psi(x)
	+\langle J(x),h\rangle
	+o(\|h\|)
	\qquad\text{as }h\to0.
	\]
	
	Let $\mathcal{M}$ be a von Neumann algebra equipped with a faithful normal
	semifinite weight $\varphi$. The associated noncommutative $L_p$-spaces are
	denoted by $L_p(\mathcal{M})$. We refer to \cite{PX03} for background on
	noncommutative integration. A detailed reduction framework for
	arbitrary von Neumann algebras is given in
	\cite{HJX10}.
	For \(1<p<\infty\), let \(q=p/(p-1)\) be the conjugate exponent of \(p\).
	The space \(L_p(\mathcal M)\), regarded as a real Banach space, is uniformly
	smooth; see \cite[Theorem~5.1 and Corollary~5.2]{PX03}. Consequently, for
	each nonzero \(x\in L_p(\mathcal M)\), there exists a unique norming
	functional \(J_p(x)\in L_q(\mathcal M)\). A direct verification shows that,
	if \(x=u|x|\) is the polar decomposition of \(x\), then the normalized
	duality map
$
	J_p:L_p(\mathcal M)\longrightarrow L_q(\mathcal M)
$
	is given by
	\[
	J_p(x)
	=
	\begin{cases}
		\|x\|_p^{\,2-p}u|x|^{p-1}, & x\neq0,\\[1mm]
		0, & x=0.
	\end{cases}
	\]
	For \(z\in L_q(\mathcal M)\) and \(y\in L_p(\mathcal M)\), we use the real
	duality pairing
	\[
	\langle z,y\rangle
	:=
	\Ree\Tr(z^*y),
	\]
	where \(\Tr\) denotes the canonical trace pairing between the Haagerup
	spaces \(L_q(\mathcal M)\) and \(L_p(\mathcal M)\). 
	
Our first result is the following sharp tangent inequality in $L_p(\mathcal M)$.
	\begin{thm}
		\label{thm:main1}
		Let \(x,y\in L_p(\mathcal M)\).
		Then for $1<p\le 2$,
		\begin{equation}\label{eq:main}
			\|x+y\|_p^2
			\geq
			\|x\|_p^2
			+2\langle J_p(x),y\rangle
			+(p-1)\|y\|_p^2.
		\end{equation}
		For $2\le p<\infty$, the inequality is reversed.
	Moreover, the coefficient \(p-1\) is sharp.
	\end{thm}
Inequality~\eqref{eq:main} reduces to Xu's tangent inequality for classical
\(L_p\)-spaces \cite{Xu89}. Indeed, when
\(\mathcal M=L_\infty(\Omega,\mu)\) is commutative and
\(x,y\in L_p(\Omega;\mathbb R)\), it reduces to
\begin{equation}\label{eq:xu}
	\|x+y\|_{L_p}^2
	\ge
	\|x\|_{L_p}^2
	+2\langle j_p(x),y\rangle
	+(p-1)\|y\|_{L_p}^2,
	\qquad 1<p\le2,
\end{equation}
where
$
\langle j_p(x),y\rangle
:=
\int_\Omega j_p(x)y\,d\mu
=
\|x\|_{L_p}^{\,2-p}
\int_\Omega |x|^{p-2}xy\,d\mu.
$
For \(2\le p<\infty\), the inequality is reversed. Thus,
Theorem~\ref{thm:main1} gives the noncommutative tangent
formulation corresponding to Xu's classical inequality. We note that Xu and Roach~\cite{XR91} subsequently placed the tangent inequality
\eqref{eq:xu} within the general theory of uniformly convex and uniformly
smooth Banach spaces. The resulting estimates, now commonly
known as the \emph{Xu--Roach inequalities}, have proved useful in
Banach-space optimization and inverse problems; see, for instance,
\cite{BKMSS08,HH14,Spr19,BBHR23}. Xu~\cite{Xu20} also explicitly emphasized their
potential relevance to the non-Euclidean analysis of machine-learning
problems. More broadly, Banach-space geometry, duality
mappings, and semi-inner-products have played an important role in the
development of learning methods in reproducing kernel Banach spaces; see
\cite{ZXZ09}.
	
	Inequality~\eqref{eq:main} may also be viewed as the tangent
	form of the sharp Ricard--Xu convexity inequality \cite[Theorem~2]{RX16}. Indeed, applying
	\eqref{eq:main} to \(y\) and \(-y\), and then adding the resulting
	inequalities yields, for every
	\(x,y\in L_p(\mathcal M)\),
\begin{equation}\label{eq:RX}
		\|x+y\|_p^2+\|x-y\|_p^2
	\ge
	2\|x\|_p^2+2(p-1)\|y\|_p^2,
	\qquad 1<p\le2,
\end{equation}
	with the reverse inequality for \(2\le p<\infty\). In fact,
\eqref{eq:main} and \eqref{eq:RX} are equivalent, as will be
	shown later in Theorem~\ref{thm:abstract-equivalence}. Thus, once
	\eqref{eq:RX} is known, Theorem~\ref{thm:abstract-equivalence} also gives
	\eqref{eq:main};
	Section~\ref{sec:proof-main} will provide a direct proof
	of the tangent inequality that does not rely on \eqref{eq:RX}, and to
	prepare for the subsequent developments concerning Bregman divergences
	and the quasi-Banach range.
	 In particular, taking
	\(\mathcal M=B(\mathcal H)\) equipped with its canonical semifinite trace,
\eqref{eq:RX}
	reduces to the celebrated Ball--Carlen--Lieb inequality for Schatten $p$-classes $S_p$
	\cite[Proposition~3]{BCL94}.
\begin{rem}
Sections~\ref{sec:conditional-expectations} and~\ref{sec:martingale}
develop structured refinements of Theorem~\ref{thm:main1}. More
precisely, the Bregman form of Theorem~\ref{thm:main1} provides the
sharp two-point estimate for the component lying in the range of a
conditional expectation, while the Ricard--Xu conditional-expectation
inequality \cite[Theorem~1]{RX16} controls the complementary residual component. Combined
with the exact Bregman--Pythagorean decomposition, these two estimates
yield the anchored conditional-expectation inequality. Iterating the
same decomposition along a filtration then gives the corresponding
martingale inequalities. Conversely, the choice \(\E=\id\) recovers
Theorem~\ref{thm:main1}, whereas the choice \(a=\E_0x\) recovers the
Ricard--Xu martingale convexity inequality \cite[Corollary~3]{RX16}.
\end{rem}

When \(0<p<1\), the space \(L_p\) is no longer a Banach space, but only
a quasi-Banach space. It is therefore natural to ask whether the
classical geometric inequalities for \(L_p\)-spaces admit meaningful
analogues in this range, possibly under additional structural
assumptions. In his ICM lecture, Xu \cite{Xu10} raised this issue explicitly in
connection with the geometry of \(L_{1/2}\).

\begin{ques}[\cite{Xu10}]\label{ques:xu-parallelogram}
	Does \(L_{1/2}\) admit an analogue of the parallelogram identity or,
	more generally, a family of characteristic inequalities that plays
	a comparable role in its quasi-Banach geometry?
\end{ques}

More recently, Zhang~\cite{Zha26} showed that Hanner's inequality for
\(S_p\) fails in general throughout the range \(0<p<1\), while a
positive-cone version remains valid under the additional order condition
\(A\pm B\ge0\). Notably, the same order condition also appears in the
Ball--Carlen--Lieb formulation of Hanner's inequality; see
\cite[Theorem~2]{BCL94}.
 This indicates that, although the classical inequalities
do not persist globally in the quasi-Banach range, part of their
geometric content may survive along suitable order segments.

Motivated by Question~\ref{ques:xu-parallelogram} and Zhang's
positive-cone result~\cite{Zha26}, we begin by showing that the formal
extension of Xu's tangent inequality \eqref{eq:xu} has no fixed direction in the range
\(0<p<1\). We then show that a sharp one-sided inequality can nevertheless
be recovered by restricting to positive order segments, first in classical
\(L_p\)-spaces and subsequently in Schatten \(p\)-classes.

\begin{prop}\label{prop:xu-tangent-no-direction}
	For every \(0<p<1\), the formal extension of
	\eqref{eq:xu} has no fixed direction. More precisely, there exist
	pairs for which each of the two possible strict inequalities occurs.
\end{prop}

Proposition~\ref{prop:xu-tangent-no-direction} shows that no global
analogue of Xu's tangent inequality \eqref{eq:xu} can hold with a prescribed direction
throughout the quasi-Banach range. We therefore impose the natural
order-segment condition \(x\pm y\ge0\), under which the relevant
first-order term remains finite.

Let \(0<p<1\), and let \(x,y\in L_p(\Omega;\mathbb R)\) satisfy
\(x\pm y\ge0\) almost everywhere on \(\Omega\). Define the
commutative positive-cone tangent form by
\begin{equation*}
	\pair{J_{L_p,+}(x)}{y}
	:=
	\begin{cases}
		\displaystyle
		\|x\|_{L_p}^{\,2-p}
		\int_{\{x>0\}}x^{p-1}y\,d\mu,
		& x\ne0,\\[3mm]
		0, & x=0.
	\end{cases}
\end{equation*}

With this notation, the sharp commutative positive-cone tangent
inequality takes the following form.

\begin{thm}\label{thm:commutative-positive-cone-tangent}
	Let \(0<p<1\), and let \(x,y\in L_p(\Omega;\mathbb R)\) satisfy
	\(x\pm y\ge0\) almost everywhere on \(\Omega\). Then
	\begin{equation*}
		\|x+y\|_{L_p}^2
		\le
		\|x\|_{L_p}^2
		+2\pair{J_{L_p,+}(x)}{y}
		+p\,2^{2/p-2}\|y\|_{L_p}^2.
	\end{equation*}
	The coefficient \(p\,2^{2/p-2}\) is optimal.
\end{thm}

Applying Theorem~\ref{thm:commutative-positive-cone-tangent} with
\(y\) and \(-y\), respectively, and adding the resulting inequalities
cancels the first-order terms and yields the following sharp midpoint
form.

\begin{cor}\label{cor:commutative-positive-cone-midpoint}
	Let \(0<p<1\), and let \(x,y\in L_p(\Omega;\mathbb R)\) satisfy
	\(x\pm y\ge0\) almost everywhere on \(\Omega\). Then
	\begin{equation*}
		\|x+y\|_{L_p}^2+\|x-y\|_{L_p}^2
		\le
		2\|x\|_{L_p}^2
		+p\,2^{2/p-1}\|y\|_{L_p}^2.
	\end{equation*}
	The coefficient \(p\,2^{2/p-1}\) is optimal.
\end{cor}
For \(0<p<1\), let \(S_p\) denote the Schatten
\(p\)-class, consisting of all compact linear operators
\(A\colon \ell_2\to\ell_2\) such that \(|A|^p\) is trace class.
The space \(S_p\) is equipped with the quasi-norm
\[
\|A\|_{S_p}
=
\bigl(\operatorname{Tr}|A|^p\bigr)^{1/p}
=
\left(\operatorname{Tr}(A^*A)^{p/2}\right)^{1/p}.
\]
We next pass from the commutative setting to Schatten \(p\)-classes.
In what follows, inequalities between self-adjoint operators are
understood in the L\"owner order. For self-adjoint \(A,B\in S_p\)
satisfying \(A\pm B\ge0\), define the Schatten positive-cone tangent
form by
\[
\pair{J_{S_p,+}(A)}{B}
:=
\frac12
\left.\frac{d}{dt}\right|_{t=0}
\|A+tB\|_{S_p}^2.
\]
As will follow from Proposition~\ref{prop:positive-schatten-plane} and
\eqref{eq:same-line-function}, this derivative exists and is finite
under the condition \(A\pm B\ge0\). In the finite-dimensional setting, let \(\mathbb H_n\) denote the
set of \(n\times n\) Hermitian matrices, and let
\(A\in\mathbb H_n\) be positive definite.
Then, for every \(B\in\mathbb H_n\), standard differentiation of trace
functions yields
\[
\pair{J_{S_p,+}(A)}{B}
=
\|A\|_{S_p}^{\,2-p}
\Ree\Tr\bigl(A^{p-1}B\bigr).
\]

Theorem~\ref{thm:commutative-positive-cone-tangent} admits the following sharp
noncommutative extension.

\begin{thm}\label{thm:positive-cone-tangent}
	Let \(0<p<1\), and let \(A,B\in S_p\) be self-adjoint operators
	satisfying \(A\pm B\ge0\). Then
	\begin{equation*}
		\|A+B\|_{S_p}^2
		\le
		\|A\|_{S_p}^2
		+2\pair{J_{S_p,+}(A)}{B}
		+p\,2^{2/p-2}\|B\|_{S_p}^2.
	\end{equation*}
	The coefficient \(p\,2^{2/p-2}\) is optimal.
\end{thm}

As in the commutative case, symmetrizing the tangent inequality produces
a midpoint inequality. More precisely, applying
Theorem~\ref{thm:positive-cone-tangent} with \(B\) and \(-B\),
respectively, and adding the resulting inequalities gives the following
corollary.

\begin{cor}\label{cor:positive-cone-midpoint}
	Let \(0<p<1\), and let \(A,B\in S_p\) be self-adjoint operators
	satisfying \(A\pm B\ge0\). Then
	\begin{equation*}
		\|A+B\|_{S_p}^2+\|A-B\|_{S_p}^2
		\le
		2\|A\|_{S_p}^2
		+p\,2^{2/p-1}\|B\|_{S_p}^2.
	\end{equation*}
	The coefficient \(p\,2^{2/p-1}\) is optimal.
\end{cor}

Corollary~\ref{cor:commutative-positive-cone-midpoint} and
Corollary~\ref{cor:positive-cone-midpoint} provide partial affirmative
answers to Question~\ref{ques:xu-parallelogram} in the commutative and
noncommutative settings, respectively. Indeed, they give sharp
one-sided substitutes for the parallelogram identity along positive
order segments.

In the commutative setting, taking \(p=1/2\) in
Corollary~\ref{cor:commutative-positive-cone-midpoint} yields
\[
\|x+y\|_{L_{1/2}}^2+\|x-y\|_{L_{1/2}}^2
\le
2\|x\|_{L_{1/2}}^2+4\|y\|_{L_{1/2}}^2,
\qquad
x\pm y\ge0
\quad\text{a.e.}
\]
Likewise, taking \(p=1/2\) in
Corollary~\ref{cor:positive-cone-midpoint} gives the Schatten-class
inequality
\[
\|A+B\|_{S_{1/2}}^2+\|A-B\|_{S_{1/2}}^2
\le
2\|A\|_{S_{1/2}}^2+4\|B\|_{S_{1/2}}^2,
\qquad
A\pm B\ge0.
\]
The coefficient \(4\) is optimal in both inequalities.

Thus, although the formal extension of Xu's tangent inequality \eqref{eq:xu} has no
fixed direction on the entire quasi-Banach space, a sharp
parallelogram-type inequality survives along positive order segments
in both classical \(L_{1/2}\)-spaces and the noncommutative Schatten
class \(S_{1/2}\). Theorems~\ref{thm:commutative-positive-cone-tangent}
and~\ref{thm:positive-cone-tangent} provide the corresponding stronger
tangent refinements, retaining both the exact first-order
term and the optimal quadratic remainder. 
\begin{rem}\label{rem:positive-cone-limitations}
	The order assumptions in
	Theorems~\ref{thm:commutative-positive-cone-tangent}
	and~\ref{thm:positive-cone-tangent} are essential.

	First, even under the positive order-segment assumption, the tangent
	inequalities in these two theorems cannot be replaced by uniform lower
	estimates. More precisely, there is no finite constant \(d_p\) such that
	\[
	\|x+y\|_{L_p}^2
	\ge
	\|x\|_{L_p}^2
	+2\pair{J_{L_p,+}(x)}{y}
	+d_p\|y\|_{L_p}^2
	\]
	holds for all \(x,y\in L_p(\Omega;\mathbb R)\) satisfying
	\(x\pm y\ge0\). Indeed, in \(\ell_p^2\), take
	\[
	x_\varepsilon=(1,\varepsilon),
	\qquad
	y_\varepsilon=(0,a\varepsilon),
	\qquad 0<a<1.
	\]
	Then \(x_\varepsilon\pm y_\varepsilon>0\), but
	\[
		\frac{
			\|x_\varepsilon+y_\varepsilon\|_p^2
			-\|x_\varepsilon\|_p^2
			-2\pair{J_{L_p,+}(x_\varepsilon)}{y_\varepsilon}
		}{\|y_\varepsilon\|_p^2}
	\sim
		\frac{2\bigl((1+a)^p-1-pa\bigr)}{pa^2}
		\varepsilon^{p-2}
		\longrightarrow-\infty.
	\]
	By identifying \(\ell_p^2\) with the diagonal subspace of
	\(S_p^2\), the same example also rules out a uniform lower analogue of
	Theorem~\ref{thm:positive-cone-tangent}.
	
	Second, the order assumptions
	\[
	x\pm y\ge0
	\qquad\text{and}\qquad
	A\pm B\ge0
	\]
	are needed to ensure that the first-order tangent terms appearing in
	Theorems~\ref{thm:commutative-positive-cone-tangent}
	and~\ref{thm:positive-cone-tangent} are finite. Without these assumptions,
	the first variation of the squared quasi-norm may be infinite. For
	example, for \(x=(1,0)\) and \(y=(0,1)\) in \(\ell_p^2\),
	\[
	\frac{\|x+ty\|_p^2-\|x\|_p^2}{t}
	=
	\frac{(1+t^p)^{2/p}-1}{t}
	\sim
	\frac2p\,t^{p-1}
	\longrightarrow\infty
	\qquad(t\downarrow0).
	\]
	Again, the corresponding diagonal matrices give the same obstruction
	in the Schatten setting.
\end{rem}

\vspace{0.1in}
\noindent\textbf{Organization of the paper.}
In Section~\ref{sec:proof-main}, we prove
Theorem~\ref{thm:main1}.
In Section~\ref{sec:equivalence-bregman}, we prove an
equivalence theorem between midpoint and tangent inequalities,
introduce the Bregman divergence associated with the squared
\(L_p\)-norm, and reformulate Theorem~\ref{thm:main1} in Bregman form.
Section~\ref{sec:conditional-expectations} establishes an exact
Bregman--Pythagorean decomposition for conditional expectations and
derives the corresponding anchored inequalities.
In Section~\ref{sec:martingale}, we iterate this decomposition along
noncommutative filtrations to obtain finite and infinite anchored
martingale inequalities.
Finally, Section~\ref{sec:quasi-banach} concerns the quasi-Banach range
\(0<p<1\), where we prove
Proposition~\ref{prop:xu-tangent-no-direction} and
Theorems~\ref{thm:commutative-positive-cone-tangent}
and~\ref{thm:positive-cone-tangent}.
	
	\section{Proof of
		Theorem~\ref{thm:main1}} \label{sec:proof-main}
	
	In this section, we give a proof of Theorem~\ref{thm:main1}.

	We begin with a scalar divided-difference estimate.
	\begin{lem}
		\label{lem:scalar-divided-difference}
	Let \(1<p\le2\) and put
$
\gamma:=2/p-1.
$		Let
	$
		\psi_p(s):=\sgn(s)|s|^{p-1},
s\in\mathbb R.
$
		For \((s,t)\ne(0,0)\), define
		\[
		k_p(s,t):=
		\begin{cases}
			\displaystyle
			\frac{\psi_p(s)-\psi_p(t)}{s-t},&s\ne t,\\[3mm]
			(p-1)|s|^{p-2},&s=t\ne0.
		\end{cases}
		\]
		Then
		\begin{equation}\label{eq:scalar-divided-difference}
			k_p(s,t)
			\ge
			(p-1)2^{\gamma}
			\bigl(|s|^p+|t|^p\bigr)^{-\gamma}.
		\end{equation}
	\end{lem}
	
	\begin{proof}
		The assertion is immediate for \(p=2\), so assume \(1<p<2\).
		First suppose that \(s\) and \(t\) have the same sign, and put
		\(a=|s|\), \(b=|t|\). By continuity, it is enough to consider
		\(a,b>0\). Then
		\[
		k_p(s,t)
		=(p-1)\int_0^1
		\bigl((1-\theta)a+\theta b\bigr)^{p-2}\,d\theta.
		\]
		Since \(u\mapsto u^{p-2}\) is convex,
		\[
		k_p(s,t)
		\ge
		(p-1)\left(\frac{a+b}{2}\right)^{p-2}.
		\]
		Moreover,
		\[
		\frac{a+b}{2}
		\le
		2^{-1/p}(a^p+b^p)^{1/p}.
		\]
		Because \(p-2\le0\), raising both sides to the power \(p-2\)
		reverses the inequality and proves
		\eqref{eq:scalar-divided-difference} in the same-sign case.
		
		Now suppose that \(s\) and \(t\) have opposite signs. Then
		\[
		k_p(s,t)=\frac{a^{p-1}+b^{p-1}}{a+b}.
		\]
		If \(a\ne b\),
		\[
			\frac{a^{p-1}+b^{p-1}}{a+b}
			-\frac{a^{p-1}-b^{p-1}}{a-b}=
			\frac{2ab\bigl(b^{p-2}-a^{p-2}\bigr)}
			{(a+b)(a-b)}
			\ge0.
		\]
		The same conclusion holds at \(a=b\) by continuity. Thus the
		opposite-sign divided difference dominates the same-sign one, and the
		first part of the proof applies.
	\end{proof}
	
	We next need a left--right interpolation estimate.
	\begin{lem}
		\label{lem:left-right-interpolation} 	Let \(1<p\le2\) and put
		$
		\gamma:=2/p-1.
		$
Let \((\mathcal A,\tau)\) be a finite von Neumann algebra, let
\(h\in\mathcal A\) be positive and invertible, and define
\(L_h(\xi)=h\xi\), \(R_h(\xi)=\xi h\), and
\(S=L_h+R_h\) on \(L_2(\mathcal A,\tau)\).
		Then
		\begin{equation}\label{eq:left-right-interpolation}
			\|S^{\gamma/2}\xi\|_p
			\le
			\bigl(2\tau(h)\bigr)^{\gamma/2}\|\xi\|_2,
			\qquad \xi\in L_2(\mathcal A,\tau).
		\end{equation}
		Consequently, for every
		\(b\in L_p(\mathcal A,\tau)\cap L_2(\mathcal A,\tau)\),
		\begin{equation}\label{eq:left-right-inverse}
			\left\langle b,S^{-\gamma}b\right\rangle_{L_2}
			\ge
			\bigl(2\tau(h)\bigr)^{-\gamma}\|b\|_p^2.
		\end{equation}
	\end{lem}
	
	\begin{proof}
		We first prove
		\begin{equation}\label{eq:left-right-endpoint}
			\|S^{1/2}\xi\|_1
			\le
			\bigl(2\tau(h)\bigr)^{1/2}\|\xi\|_2.
		\end{equation}
		By \(L_1\)--\(L_\infty\) duality and the self-adjointness of
		\(S^{1/2}\) on \(L_2\),
		\[
		\begin{aligned}
			\|S^{1/2}\xi\|_1
			&=
			\sup_{\|u\|_\infty\le1}
			\left|\left\langle u,S^{1/2}\xi\right\rangle_{L_2}\right|\\
			&=
			\sup_{\|u\|_\infty\le1}
			\left|\left\langle S^{1/2}u,\xi\right\rangle_{L_2}\right|\\
			&\le
			\sup_{\|u\|_\infty\le1}
			\|S^{1/2}u\|_2\,\|\xi\|_2.
		\end{aligned}
		\]
		Furthermore,
		\[
			\|S^{1/2}u\|_2^2
			=\tau(u^*hu)+\tau(u^*uh)
			\le2\tau(h)\|u\|_\infty^2.
		\]
		This proves \eqref{eq:left-right-endpoint}.
		
		Consider the analytic family \(T_z=S^{z/2}\) on the strip
		\(0\le \Re z\le1\). Since \(S\) is positive and invertible on
		\(L_2(\mathcal A,\tau)\), \(T_{it}=S^{it/2}\) is unitary on \(L_2\).
		Moreover,
		\[
		T_{1+it}=S^{1/2}S^{it/2},
		\]
		and hence \eqref{eq:left-right-endpoint} gives
		\[
		\|T_{1+it}\xi\|_1
		\le \bigl(2\tau(h)\bigr)^{1/2}\|S^{it/2}\xi\|_2
		=\bigl(2\tau(h)\bigr)^{1/2}\|\xi\|_2.
		\]
		Since
		\[
		\frac1p=\frac{1-\gamma}{2}+\gamma,
		\qquad
		[L_2(\mathcal A,\tau),L_1(\mathcal A,\tau)]_\gamma
		=L_p(\mathcal A,\tau),
		\]
	standard	complex interpolation yields \eqref{eq:left-right-interpolation}.
		Applying it to
		\(\xi=S^{-\gamma/2}b\) and squaring gives
		\eqref{eq:left-right-inverse}.
	\end{proof}
	
	These two lemmas yield the following sharp tracial Hessian estimate.
	\begin{prop}\label{prop:derivative} 	Let \(1<p\le2\) and put
		$
		\gamma:=2/p-1.
		$
Let \((\mathcal A,\tau)\) be a finite von Neumann algebra. Let
\(A=A^*\in\mathcal A\) be invertible and let \(B=B^*\in\mathcal A\).
Set \(F(t):=\tau(|A+tB|^p)\) and
\(\Phi(t):=\|A+tB\|_p^2=F(t)^{2/p}\).
		At every \(t\) for which \(A+tB\) is invertible,
		\begin{align}
			F''(t)
			&\ge
			p(p-1)\|A+tB\|_p^{p-2}\|B\|_p^2,
			\label{eq:tracial-F-hessian}\\
			\Phi''(t)
			&\ge
			2(p-1)\|B\|_p^2.
			\label{eq:tracial-Phi-hessian}
		\end{align}
	\end{prop}
	
	\begin{proof}
	It is enough to prove the assertions at \(t=0\). Since
	\(0\notin\sigma(A)\), choose disjoint open neighborhoods
	\(U_+\supset\sigma(A)\cap(0,\infty)\) and
	\(U_-\supset\sigma(A)\cap(-\infty,0)\). On \(U_+\), define
	\(\psi_p(z)=z^{p-1}\), and on \(U_-\), define
	\(\psi_p(z)=-(-z)^{p-1}\), using holomorphic branches on the two
	components. Thus \(\psi_p\) is holomorphic on
	\(U_+\cup U_-\), an open neighborhood of \(\sigma(A)\). The
	Cauchy-integral formula for Fr\'echet derivatives in the holomorphic
	functional calculus gives the standard divided-difference identity
		\[
		D\psi_p(A)[B]=k_p(L_A,R_A)B,
		\]
		where \(L_A\) and \(R_A\) are the commuting self-adjoint left and right
		multiplication operators on \(L_2(\mathcal A,\tau)\), and the right-hand
		side is defined by their joint functional calculus. Hence trace
		differentiation gives
		\begin{equation}\label{eq:trace-second-variation}
			F''(0)
			=p\left\langle B,k_p(L_A,R_A)B\right\rangle_{L_2}.
		\end{equation}
		When \(A=\sum_i\lambda_i e_i\) has finite spectrum,
		\eqref{eq:trace-second-variation} reads
		\[
		F''(0)
		=p\sum_{i,j}k_p(\lambda_i,\lambda_j)
		\tau(e_iBe_jB).
		\]
		For completeness, choose finite-spectrum self-adjoint operators
		\(A_n=f_n(A)\) such that \(\|A_n-A\|_\infty\to0\) and
	$
		\inf_n\operatorname{dist}(0,\sigma(A_n))>0.
	$
		The function \(k_p\) is continuous on a compact neighborhood of
		\(\sigma(A)\times\sigma(A)\); therefore, the joint functional calculus
		gives
		\[
		k_p(L_{A_n},R_{A_n})
		\longrightarrow k_p(L_A,R_A)
		\quad\text{in }B(L_2(\mathcal A,\tau)).
		\]
		Thus the finite-spectrum identity passes to the limit and yields
		\eqref{eq:trace-second-variation} for general invertible \(A\).
		
		By Lemma~\ref{lem:scalar-divided-difference} and the joint functional
		calculus,
		\[
		k_p(L_A,R_A)
		\ge
		(p-1)2^\gamma
		\bigl(L_{|A|^p}+R_{|A|^p}\bigr)^{-\gamma}.
		\]
		Apply Lemma~\ref{lem:left-right-interpolation} with \(h=|A|^p\).
		Since \(\tau(h)=\|A\|_p^p\),
		\[
		\begin{aligned}
			F''(0)
			&\ge
			p(p-1)2^\gamma
			\left\langle B,(L_h+R_h)^{-\gamma}B\right\rangle_{L_2}\\
			&\ge
			p(p-1)2^\gamma(2\tau(h))^{-\gamma}\|B\|_p^2\\
			&=
			p(p-1)\|A\|_p^{p-2}\|B\|_p^2.
		\end{aligned}
		\]
		This proves \eqref{eq:tracial-F-hessian}.
		
		Finally,
		\[
		\frac12\Phi''(0)
		=
		\frac1pF(0)^{2/p-1}F''(0)
		+\frac{2-p}{p^2}F(0)^{2/p-2}F'(0)^2.
		\]
		The second term is nonnegative. Since \(F(0)=\|A\|_p^p\),
		\eqref{eq:tracial-F-hessian} gives
		\[
		\frac12\Phi''(0)
		\ge(p-1)\|B\|_p^2,
		\]
		which is \eqref{eq:tracial-Phi-hessian}.
	\end{proof}
	
	We now obtain the finite tracial tangent inequality.
	\begin{prop}
		\label{prop:finite-tracial-tangent}
		Let \((\mathcal A,\tau)\) be a finite von Neumann algebra and let
		\(x,y\in L_p(\mathcal A,\tau)\). Then for $1<p\le 2$, 
		\begin{equation}\label{eq:finite-tracial-tangent}
			\|x+y\|_p^2
			\ge
			\|x\|_p^2
			+2\pair{J_p(x)}{y}
			+(p-1)\|y\|_p^2.
		\end{equation}
	\end{prop}
	
	\begin{proof}
		Assume first that \(x,y\in\mathcal A\). On
		\(M_2(\mathcal A)\), equipped with
		\(\widetilde\tau=\tau\otimes\mathrm{Tr}_2\), define
		\[
		X_\varepsilon(t)
		:=
		\begin{pmatrix}
			\varepsilon1&x+ty\\
			(x+ty)^*&-\varepsilon1
		\end{pmatrix},
		\qquad
		Y:=
		\begin{pmatrix}
			0&y\\
			y^*&0
		\end{pmatrix}.
		\]
		For every \(\varepsilon>0\),
		\[
		X_\varepsilon(t)^2
		=
		\begin{pmatrix}
			(x+ty)(x+ty)^*+\varepsilon^2 1&0\\
			0&(x+ty)^*(x+ty)+\varepsilon^2 1
		\end{pmatrix}
		\ge\varepsilon^2 1.
		\]
		Thus Proposition~\ref{prop:derivative}, applied at every point of
		the line, shows that
		\[
		\phi_\varepsilon''(t)
		\ge2(p-1)\|Y\|_p^2,
		\qquad
		\phi_\varepsilon(t):=\|X_\varepsilon(t)\|_p^2.
		\]
		Integrating twice on \([0,1]\) yields
		\begin{equation}\label{eq:regularized-block-tangent}
			\phi_\varepsilon(1)
			\ge
			\phi_\varepsilon(0)+\phi_\varepsilon'(0)
			+(p-1)\|Y\|_p^2.
		\end{equation}
		
		Put
	$
		H_\varepsilon(z)
		:=\tau\bigl((z^*z+\varepsilon^2 1)^{p/2}\bigr).
		$
		The polar decomposition and traciality give
		\[
		\phi_\varepsilon(t)
		=2^{2/p}H_\varepsilon(x+ty)^{2/p},
		\qquad
		\|Y\|_p^2=2^{2/p}\|y\|_p^2.
		\]
		Trace differentiation gives
		\begin{equation}\label{eq:regularized-block-derivative}
				\phi_\varepsilon'(0)
				=2^{2/p+1}H_\varepsilon(x)^{2/p-1}\times
				\Ree\tau\bigl(
				(x^*x+\varepsilon^2 1)^{p/2-1}x^*y
				\bigr).
		\end{equation}
		If \(x=u|x|\), then
		\[
		u|x|(|x|^2+\varepsilon^2 1)^{p/2-1}
		\longrightarrow u|x|^{p-1}
		\quad\text{in }L_q,
		\qquad q=\frac{p}{p-1}.
		\]
		Indeed, the scalar factor on the left is bounded by \(|x|^{p-1}\),
		whose \(q\)-th power is \(|x|^p\). Hence dominated convergence in
		\(L_q\) applies. Consequently, \eqref{eq:regularized-block-derivative} yields
		\[
		\lim_{\varepsilon\downarrow0}\phi_\varepsilon'(0)
		=2^{2/p+1}\pair{J_p(x)}{y}.
		\]
		Also,
		\[
		\lim_{\varepsilon\downarrow0}\phi_\varepsilon(t)
		=2^{2/p}\|x+ty\|_p^2.
		\]
		Letting \(\varepsilon\downarrow0\) in
		\eqref{eq:regularized-block-tangent} and dividing by \(2^{2/p}\)
		proves \eqref{eq:finite-tracial-tangent} for bounded \(x,y\).
		
		For general \(x,y\in L_p(\mathcal A,\tau)\), approximate them in
		\(L_p\) by bounded elements. The norms are continuous and the normalized
		duality map \(J_p:L_p\to L_q\) is norm-to-norm continuous, so
		\eqref{eq:finite-tracial-tangent} passes to the limit.
	\end{proof}
	
	We next pass from finite tracial spaces to arbitrary Haagerup \(L_p\)-spaces.
		\begin{lem}
		\label{lem:haagerup-transfer}
Let $1<p\le 2$.		If \eqref{eq:finite-tracial-tangent} holds in every finite tracial
		noncommutative \(L_p\)-space, then it holds in \(L_p(\M)\) for every
		von Neumann algebra \(\M\).
	\end{lem}
	
	\begin{proof}
		Fix \(x,y\in L_p(\M)\). The left and right support projections of each
		element of \(L_p(\M)\) are \(\sigma\)-finite. A finite join of
		\(\sigma\)-finite projections is again \(\sigma\)-finite. Hence the join
		of the four support projections associated with \(x\) and \(y\) is a
		\(\sigma\)-finite projection \(e\), and \(x,y\in L_p(e\M e)\).
		By \cite[Remark~3.2]{HJX10}, under the canonical isometric identification
		\(L_p(e\M e)=eL_p(\M)e\), both the norm and the duality pairing in
		\eqref{eq:main} are unchanged. Thus it is enough to consider the
		\(\sigma\)-finite case. By \cite[Theorem~3.1]{HJX10}, there exist
		\(X_p=L_p(\mathcal R)\), an isometric linear embedding
		\(\iota:L_p(\M)\to X_p\), and an increasing sequence
		\((X_{p,n})_{n\ge1}\) of isometric copies of finite tracial
		\(L_p\)-spaces whose union is dense in \(X_p\).
		
		Choose \(\widetilde x_n,\widetilde y_n\in\bigcup_m X_{p,m}\) such that
		\[
		\widetilde x_n\longrightarrow\iota(x),
		\qquad
		\widetilde y_n\longrightarrow\iota(y)
		\quad\text{in }X_p.
		\]
		Since the sequence is increasing, for each \(n\) we may choose a single
		index \(m_n\) such that
		\(\widetilde x_n,\widetilde y_n\in X_{p,m_n}\).
		Put \(x_n=\widetilde x_n\) and \(y_n=\widetilde y_n\).
		For \(\Psi(\xi)=\frac12\|\xi\|_{X_p}^2\), the finite tracial estimate
		on \(X_{p,m_n}\) can be written intrinsically as
		\[
		\|x_n+y_n\|_{X_p}^2
		\ge
		\|x_n\|_{X_p}^2
		+2D\Psi(x_n)[y_n]
		+(p-1)\|y_n\|_{X_p}^2.
		\]
		Indeed, the norm on \(X_{p,m_n}\) is the restriction of the norm on
		\(X_p\), so the derivative of \(\Psi\) restricted to \(X_{p,m_n}\)
		is exactly the normalized duality functional of that finite tracial
		\(L_p\)-space. Since \(1<p<\infty\), the normalized duality map of
		\(X_p\) is norm-to-norm continuous. Passing to the limit gives the same
		inequality for \(\iota(x),\iota(y)\). Finally, because \(\iota\) is a
		linear isometry, the restriction of \(D\Psi(\iota(x))\) to
		\(\iota(L_p(\M))\) is the normalized duality functional at \(x\); hence
		\[
		D\Psi(\iota(x))[\iota(y)]
		=\pair{J_p(x)}{y}.
		\]
		This proves the assertion.
	\end{proof}
	
	\begin{proof}[Proof of Theorem~\ref{thm:main1}]
		For \(1<p\le2\), Proposition~\ref{prop:finite-tracial-tangent} and
		Lemma~\ref{lem:haagerup-transfer} give \eqref{eq:main} directly.
		
		Now let \(2\le p<\infty\), and put \(q=p/(p-1)\). Then
		\(1<q\le2\) and \(q-1=1/(p-1)\). Set
	$
		\Phi_r(z):=\frac12\|z\|_r^2.
	$
		The lower estimate already proved in \(L_q(\M)\) says
		\begin{equation}\label{eq:q-strong-convexity}
			\Phi_q(z+w)
			\ge
			\Phi_q(z)+\pair{J_q(z)}{w}
			+\frac{q-1}{2}\|w\|_q^2.
		\end{equation}
		Fix \(x,y\in L_p(\M)\) and put \(z=J_p(x)\). A direct calculation
		from the polar decomposition gives
		\[
		\|z\|_q=\|x\|_p,
		\qquad
		J_q(z)=x.
		\]
		For \(\eta=z+w\), \eqref{eq:q-strong-convexity} gives
		\[
		\Phi_q(\eta)
		\ge
		\Phi_q(z)+\pair{x}{w}
		+\frac{q-1}{2}\|w\|_q^2.
		\]
		The elementary dual identity
		\[
		\Phi_p(v)
		=
		\sup_{\eta\in L_q(\M)}
		\bigl\{\pair{\eta}{v}-\Phi_q(\eta)\bigr\}
		\]
		follows from H\"older's inequality and scalar optimization, with
		equality at \(\eta=J_p(v)\). Therefore,
		\[
		\begin{aligned}
			\Phi_p(x+y)
			&\le
			\pair{z}{x+y}-\Phi_q(z)
			+\sup_{w\in L_q(\M)}
			\left\{
			\pair{w}{y}-\frac{q-1}{2}\|w\|_q^2
			\right\}\\
			&=
			\Phi_p(x)+\pair{J_p(x)}{y}
			+\frac{1}{2(q-1)}\|y\|_p^2\\
			&=
			\Phi_p(x)+\pair{J_p(x)}{y}
			+\frac{p-1}{2}\|y\|_p^2.
		\end{aligned}
		\]
		Multiplying by \(2\) proves the reverse of \eqref{eq:main}.
		
		It remains to prove optimality. Take
		\(\M=\mathbb C^2\) with normalized trace
		\(\tau(a,b)=(a+b)/2\), and let
$
		x=(1,1)$ and
	$	h=(1,-1).$
		Then
		\[
		\|x\|_p=\|h\|_p=1,
		\qquad
		\pair{J_p(x)}{h}=0,
		\]
		and, as \(t\to0\),
		\[
			\|x+th\|_p^2
			=
			\left(
			\frac{(1+t)^p+(1-t)^p}{2}
			\right)^{2/p}\\
			=1+(p-1)t^2+O(t^4).
		\]
		Thus any uniform coefficient in the lower estimate must be at most
		\(p-1\), whereas any uniform coefficient in the upper estimate must be
		at least \(p-1\). Hence \(p-1\) is optimal in both ranges.
	\end{proof}
	
\section{An equivalence theorem and the Bregman formulation}
\label{sec:equivalence-bregman}

In this section, we establish an equivalence theorem showing that the
convexity inequality and its tangent formulation are equivalent in a
precise sense. We also introduce the associated Bregman divergence and
use it to give a clean equivalent reformulation of
Theorem~\ref{thm:main1}.

	We recall the relevant differentiability notion. Let \(X\) be a
	real Banach space and let \(F:X\to\mathbb R\). We say that \(F\) is
	\emph{G\^ateaux differentiable} at \(x\in X\) if there exists a
	continuous linear functional \(DF(x)\in X^*\) such that
	\begin{equation*}
		DF(x)[h]
		=
		\lim_{t\to0}
		\frac{F(x+th)-F(x)}{t},
		\qquad h\in X.
	\end{equation*}
	Thus, G\^ateaux differentiability amounts to differentiability along
	every affine line through \(x\), with the resulting directional
	derivatives depending continuously and linearly on the direction.
	
	Recall that Fr\'echet differentiability is stronger: \(F\) is
	\emph{Fr\'echet differentiable} at \(x\) if there exists \(DF(x)\in X^*\)
	such that
	\[
	F(x+h)
	=
	F(x)+DF(x)[h]+o(\|h\|)
	\qquad\text{as }h\to0.
	\]
	In particular, Fr\'echet differentiability implies G\^ateaux
	differentiability, whereas the converse does not hold in general.
	For the applications below, the squared norm on \(L_p(\mathcal M)\)
	is in fact Fr\'echet differentiable because \(L_p(\mathcal M)\) is
	uniformly smooth. 
	
	The following equivalence theorem requires only G\^ateaux differentiability and will play a central role in the arguments that follow.
	
\begin{thm}\label{thm:abstract-equivalence}
Let \(X\) be a real Banach space such that the function
\(\Phi:X\to\mathbb R\), defined by
\(\Phi(x):=\frac12\|x\|^2\), is G\^ateaux differentiable. Let \(c\geq0\). Then
	\begin{equation}\label{eq:abstract-mid-lower}
		\|x+y\|^2+\|x-y\|^2
		\geq
		2\|x\|^2+2c\|y\|^2,
		\qquad x,y\in X,
	\end{equation}
	if and only if
	\begin{equation}\label{eq:abstract-tangent-lower}
		\|x+y\|^2
		\geq
		\|x\|^2+2D\Phi(x)[y]+c\|y\|^2,
		\qquad x,y\in X.
	\end{equation}
	The analogous equivalence holds when both inequalities are reversed.
	Moreover, the two formulations have the same admissible constants
	\(c\), and hence the same optimal coefficient.
\end{thm}
	
	\begin{proof}
		Assume first that \eqref{eq:abstract-mid-lower} holds. Fix
		\(x,y\in X\), and define
		\begin{equation*}
			g(t)
			:=
			\Phi(x+ty)-\frac{c}{2}t^2\|y\|^2,
			\qquad t\in\mathbb R.
		\end{equation*}
		For \(s,t\in\mathbb R\), apply
		\eqref{eq:abstract-mid-lower} with \(x+ty\) in place of \(x\)
		and \(sy\) in place of \(y\). After dividing by \(2\), one obtains
		\[
		\Phi(x+(t+s)y)+\Phi(x+(t-s)y)
		\geq
		2\Phi(x+ty)+c s^2\|y\|^2.
		\]
		Using
		$
		(t+s)^2+(t-s)^2=2t^2+2s^2,
		$
		this is equivalent to
		\[
		g(t+s)+g(t-s)\geq2g(t).
		\]
		Thus \(g\) is midpoint convex. Since \(g\) is continuous, it is
		convex.
		
		By the G\^ateaux differentiability of \(\Phi\), the scalar function
		\(g\) is differentiable at \(0\), with
		$
		g'(0)=D\Phi(x)[y].
		$
		The supporting-line inequality for the convex function \(g\)
		therefore gives
		\[
		g(1)\geq g(0)+g'(0).
		\]
		Substituting the definition of \(g\), we obtain
		\[
		\Phi(x+y)-\frac{c}{2}\|y\|^2
		\geq
		\Phi(x)+D\Phi(x)[y].
		\]
		Multiplying by \(2\) yields
		\eqref{eq:abstract-tangent-lower}.
		
		Conversely, apply \eqref{eq:abstract-tangent-lower} to
		\((x,y)\) and \((x,-y)\). Since \(D\Phi(x)\in X^*\) is linear,
		\[
		D\Phi(x)[-y]=-D\Phi(x)[y].
		\]
		Adding the two resulting inequalities therefore cancels the
		first-order terms and gives \eqref{eq:abstract-mid-lower}.
		
		If both inequalities are reversed, the same argument shows that
		\(g\) is midpoint concave and hence concave. Its supporting-line
		inequality is then reversed, while the converse again follows by
		symmetrization.
		
		Since the equivalence holds for every fixed \(c\geq0\), the sets of
		admissible constants in the midpoint and tangent formulations coincide.
		Consequently, the optimal coefficients in the two formulations are the
		same.
	\end{proof}
As an immediate consequence of
Theorem~\ref{thm:abstract-equivalence}, we have the following.

\begin{cor}\label{cor:main-rx-equivalence}
	For \(1<p\le2\), inequality~\eqref{eq:main} is equivalent to
	inequality~\eqref{eq:RX}. For \(2\le p<\infty\), the corresponding
	reversed inequalities are equivalent.
\end{cor}
	
		Let \(x=u|x|\in L_p(\mathcal M)\) be the polar decomposition of
	\(x\). Recall the normalized duality map
	$
	J_p:L_p(\mathcal M)\longrightarrow L_q(\mathcal M)
	$
	is given by
	\begin{equation*}
		J_p(x)
		=
		\begin{cases}
			\|x\|_p^{\,2-p}u|x|^{p-1},
			& x\neq0,\\[2mm]
			0,
			& x=0.
		\end{cases}
	\end{equation*}
	That is,
	for every \(x\in L_p(\mathcal M)\),
	$
	\|J_p(x)\|_q
	=
	\|x\|_p
	$ and $
	\pair{J_p(x)}{x}
	=
	\|x\|_p^2.
	$

	Following Bregman~\cite{Bre67}, we use the normalized duality map to
	define the Bregman divergence associated with the squared \(L_p\)-norm. Set
	\begin{equation}\label{eq:Phi-p}
		\Phi_p:L_p(\mathcal M)\longrightarrow\mathbb R,
		\qquad
		\Phi_p(x):=\frac12\|x\|_p^2.
	\end{equation}
	Regarded as a function on the underlying real Banach space of
	\(L_p(\mathcal M)\), the functional \(\Phi_p\) is Fr\'echet
	differentiable, and
	\begin{equation}\label{eq:derivative}
		D\Phi_p(a)[h]
		=
		\pair{J_p(a)}{h},
		\qquad
		a,h\in L_p(\mathcal M).
	\end{equation}
	By \eqref{eq:derivative}, the first-order affine approximation of \(\Phi_p\) at
	\(a\in L_p(\mathcal M)\) is
	\[
	x\longmapsto
	\Phi_p(a)+\pair{J_p(a)}{x-a}.
	\]
	The difference between \(\Phi_p(x)\) and this affine approximation
	defines the corresponding Bregman divergence.
	
	\begin{defin}\label{def:bregman}
		Let \(1<p<\infty\). For \(x,a\in L_p(\mathcal M)\), define
		the Bregman divergence generated by \(\Phi_p\) in \eqref{eq:Phi-p} by
		\begin{equation}\label{eq:bregman}
			\begin{aligned}
				\mathcal D_p(x\mid a)
				&:=
				\Phi_p(x)-\Phi_p(a)-D\Phi_p(a)[x-a]\\
				&=
				\frac12\|x\|_p^2
				-\frac12\|a\|_p^2
				-\pair{J_p(a)}{x-a}.
			\end{aligned}
		\end{equation}
	\end{defin}
	
	Since \(\Phi_p\) is strictly convex, one has
	$
	\mathcal D_p(x\mid a)\geq0,
	$
	with equality if and only if \(x=a\). In general,
	\(\mathcal D_p\) is not symmetric; namely,
	$
	\mathcal D_p(x\mid a)
	\neq
	\mathcal D_p(a\mid x).
	$
	The order of the variables is therefore essential: the second
	variable \(a\) is the base point at which the tangent functional
	\(J_p(a)\) is evaluated.

	In terms of the Bregman divergence introduced in Definition~\ref{def:bregman}, Theorem~\ref{thm:main1} has the following clean form.
	
	\begin{cor}\label{cor:bregman-two-point}
		Let $x,a\in L_p(\M)$. Then for $1<p\le 2$,
		\begin{equation*}
			\mathcal D_p(x\mid a)
			\ge \frac{p-1}{2}\|x-a\|_p^2.
		\end{equation*}
		For $2\le p<\infty$, the inequality is reversed.
	\end{cor}
	\begin{proof}
		Apply Theorem~\ref{thm:main1} with base point \(a\) and increment
		\(x-a\), and divide by \(2\).
	\end{proof}
\section{Conditional expectations and Bregman decomposition}
\label{sec:conditional-expectations}
In this section, we develop a conditional-expectation refinement of
Theorem~\ref{thm:main1}. We first establish the relevant duality
orthogonality and derive an exact Bregman--Pythagorean decomposition
with respect to a conditional expectation. Combining this decomposition
with the two-point Bregman form of Theorem~\ref{thm:main1} and the
Ricard--Xu conditional-expectation inequality \cite[Theorem~1]{RX16}, we obtain a sharp
anchored estimate that separately controls the component in
\(L_p(\N)\) and the residual component orthogonal to it. We also give a
perturbative formulation, prove the independent optimality of both
quadratic coefficients, and recover Theorem~\ref{thm:main1} and the
Ricard--Xu inequality \eqref{eq:RX} as special cases.

Let $\M$ be equipped with a faithful normal semifinite weight $\varphi$, and let
$\N\subseteq\M$ be a von Neumann subalgebra such that
$\varphi|_{\N}$ is semifinite and $\N$ is invariant under the modular group
$\sigma^\varphi$. By Takesaki's theorem~\cite[pp.~306--310]{Tak72}, there is a
unique $\varphi$-preserving normal faithful conditional expectation
$
	\E:\M\longrightarrow\N.
$
For every \(1\le r<\infty\), the map \(\E\) extends to a positive
contractive projection from \(L_r(\M)\) onto \(L_r(\N)\); see
\cite[Example~5.8]{HJX10}. We use the same symbol for all these
extensions.

The following duality relation will be used repeatedly.

\begin{lem}\label{lem:expectation-duality}
	Let $1<p<\infty$ and $q=p/(p-1)$. For $z\in L_q(\N)$ and
	$x\in L_p(\M)$,
	\begin{equation}\label{eq:expectation-duality}
		\Tr(z^*x)=\Tr(z^*\E x).
	\end{equation}
	Consequently, if $a\in L_p(\N)$, then
	\begin{equation}\label{eq:tangent-orthogonality}
		\pair{J_p(a)}{x-\E x}=0.
	\end{equation}
\end{lem}

%\begin{proof}
%	Under Haagerup duality, the inclusion
%	$L_q(\N)\hookrightarrow L_q(\M)$ is the Banach adjoint of
%	$\E:L_p(\M)\to L_p(\N)$; see~\cite[p.~2144]{HJX10}. This gives
%	\eqref{eq:expectation-duality}. Since $a\in L_p(\N)$ implies
%	$J_p(a)\in L_q(\N)$, applying \eqref{eq:expectation-duality} to
%	$x-\E x$ gives \eqref{eq:tangent-orthogonality}.
%\end{proof}
\begin{proof}
	Under Haagerup duality, the inclusion
	$L_q(\N)\hookrightarrow L_q(\M)$ is the Banach adjoint of
	$\E:L_p(\M)\to L_p(\N)$; see
	\cite[Proposition~5.4(ii) and Example~5.8]{HJX10}. This gives
	\eqref{eq:expectation-duality}. Since $a\in L_p(\N)$ implies
	$J_p(a)\in L_q(\N)$, applying \eqref{eq:expectation-duality} to
	$x-\E x$ gives \eqref{eq:tangent-orthogonality}.
\end{proof}

We next identify the exact structural identity underlying the anchored estimates. It is a Bregman--Pythagorean decomposition with respect to the conditional expectation \(\E\).

\begin{thm}\label{thm:bregman-pythagorean}
	For every \(x\in L_p(\M)\) and \(a\in L_p(\N)\),
	\begin{equation}\label{eq:bregman-pythagorean}
		\mathcal D_p(x\mid a)
		=
		\mathcal D_p(x\mid\E x)
		+
		\mathcal D_p(\E x\mid a).
	\end{equation}
	Moreover,
	\begin{equation}\label{eq:bregman-residual}
		\mathcal D_p(x\mid\E x)
		=
		\frac12\bigl(\|x\|_p^2-\|\E x\|_p^2\bigr).
	\end{equation}
\end{thm}

\begin{proof}
	A direct expansion of \eqref{eq:bregman} gives the three-point
	identity
	\begin{equation}\label{eq:bregman-three-point}
		\mathcal D_p(x\mid a)-\mathcal D_p(x\mid z)
		-\mathcal D_p(z\mid a)
		=\pair{J_p(z)-J_p(a)}{x-z}.
	\end{equation}
	Put $z=\E x$. Both $J_p(z)$ and $J_p(a)$ belong to $L_q(\N)$, so the
	right-hand side of \eqref{eq:bregman-three-point} vanishes by
	Lemma~\ref{lem:expectation-duality}. This proves
	\eqref{eq:bregman-pythagorean}. Using the defining formula
	\eqref{eq:bregman} with base point \(\E x\), together with
	\eqref{eq:tangent-orthogonality}, gives \eqref{eq:bregman-residual}.
\end{proof}

\begin{rem}\label{rem:pythagorean-exact}
	The decomposition \eqref{eq:bregman-pythagorean} is exact for every
	$1<p<\infty$. The change of direction at $p=2$ enters only through the
	quadratic estimates applied to its two summands.
\end{rem}

We shall use the conditional-expectation inequality of Ricard and
Xu~\cite[Theorem~1]{RX16}.

\begin{lem}[\cite{RX16}]\label{lem:ricard-xu-expectation}
	Let $x\in L_p(\M)$. If $1<p\le2$, then
	\begin{equation*}
		\|x\|_p^2
		\ge \|\E x\|_p^2+(p-1)\|x-\E x\|_p^2.
	\end{equation*}
	If $2\le p<\infty$, the inequality is reversed.
\end{lem}

\begin{thm}\label{thm:anchored-expectation}
	Let \(a\in L_p(\N)\) and \(x\in L_p(\M)\). If \(1<p\le2\), then
	\begin{equation}\label{eq:anchored-expectation}
		\mathcal D_p(x\mid a)
		\ge
		\frac{p-1}{2}
		\left(
		\|\E x-a\|_p^2
		+
		\|x-\E x\|_p^2
		\right).
	\end{equation}
	If \(2\le p<\infty\), the inequality is reversed. Moreover, the
	coefficients of \(\|\E x-a\|_p^2\) and \(\|x-\E x\|_p^2\) are
	independently optimal.
\end{thm}

\begin{proof}
	Assume first that \(1<p\le2\). By
	Theorem~\ref{thm:bregman-pythagorean},
	\[
	\mathcal D_p(x\mid a)
	=
	\mathcal D_p(x\mid\E x)
	+
	\mathcal D_p(\E x\mid a).
	\]
	Using \eqref{eq:bregman-residual} together with
	Lemma~\ref{lem:ricard-xu-expectation}, we obtain
	\[
		\mathcal D_p(x\mid\E x)
		=
		\frac12
		\bigl(\|x\|_p^2-\|\E x\|_p^2\bigr)\ge
		\frac{p-1}{2}\|x-\E x\|_p^2.
	\]
	Since \(a,\E x\in L_p(\N)\), Corollary~\ref{cor:bregman-two-point},
	applied in \(L_p(\N)\), gives
	\[
	\mathcal D_p(\E x\mid a)
	\ge
	\frac{p-1}{2}\|\E x-a\|_p^2.
	\]
	Adding these two estimates yields
	\[
	\mathcal D_p(x\mid a)
	\ge
	\frac{p-1}{2}
	\left(
	\|\E x-a\|_p^2
	+
	\|x-\E x\|_p^2
	\right),
	\]
	which proves \eqref{eq:anchored-expectation}.
	
	For \(2\le p<\infty\), both
	Lemma~\ref{lem:ricard-xu-expectation} and
	Corollary~\ref{cor:bregman-two-point} hold with the reverse
	inequality, whereas the Bregman--Pythagorean decomposition remains
	exact. The same argument therefore yields the reverse of
	\eqref{eq:anchored-expectation}.

	It remains to verify the optimality of the two quadratic
coefficients. We show that they are independently sharp.

For the coefficient of \(\|\E x-a\|_p^2\), take
\(\N=\M\) and \(\E=\id\). Then the residual term vanishes, and
\eqref{eq:anchored-expectation} reduces to the sharp two-point
estimate
\[
\mathcal D_p(x\mid a)
\ge
\frac{p-1}{2}\|x-a\|_p^2,
\qquad 1<p\le2,
\]
with the reverse inequality for \(2\le p<\infty\). Its optimality
follows from Theorem~\ref{thm:main1}.

To prove the optimality of the coefficient of
\(\|x-\E x\|_p^2\), consider
\(\M=\ell_\infty^2\) equipped with the normalized trace
$
\tau(\alpha,\beta)=(\alpha+\beta)/2.
$
Let \(\N=\mathbb C(1,1)\), and let
$
\E(\alpha,\beta)
=
\left(
(\alpha+\beta)/2,
(\alpha+\beta)/2
\right).
$
For sufficiently small \(t\in\mathbb R\), set
$
a=(1,1),
$ and $
x_t=(1+t,1-t).
$
Then
\[
\E x_t=a,
\qquad
\|x_t-\E x_t\|_p^2=t^2.
\]
Moreover, by \eqref{eq:bregman-residual},
\[
\begin{aligned}
	\mathcal D_p(x_t\mid a)
	&=
	\frac12\bigl(\|x_t\|_p^2-\|a\|_p^2\bigr)\\
	&=
	\frac12
	\left[
	\left(
	\frac{(1+t)^p+(1-t)^p}{2}
	\right)^{2/p}
	-1
	\right]\\
	&=
	\frac{p-1}{2}t^2+O(t^4).
\end{aligned}
\]
If the coefficient \((p-1)/2\) of the residual term were replaced by
\(C\), then the lower estimate for \(1<p\le2\) would force
\(C\le (p-1)/2\), whereas the upper estimate for \(2\le p<\infty\)
would force \(C\ge (p-1)/2\). Hence \((p-1)/2\) is optimal in both
ranges.
\end{proof}

The anchored theorem may also be written in perturbative form.

\begin{cor}\label{cor:expectation-perturbation}
	Let $a\in L_p(\N)$ and $y\in L_p(\M)$. If $1<p\le2$, then
	\begin{equation*}
			\|a+y\|_p^2
			\ge \|a\|_p^2+2\pair{J_p(a)}{\E y}
			+(p-1)\|\E y\|_p^2 +(p-1)\|y-\E y\|_p^2.
	\end{equation*}
	If $2\le p<\infty$, the inequality is reversed. Moreover,
	$\pair{J_p(a)}{\E y}=\pair{J_p(a)}{y}$.
\end{cor}

\begin{proof}
	Apply Theorem~\ref{thm:anchored-expectation} to $x=a+y$ and use
	$\E a=a$. The last assertion follows from
	Lemma~\ref{lem:expectation-duality}.
\end{proof}

\begin{cor}\label{cor:recoveries}
	The following special cases hold.
	\begin{enumerate}[label=\textup{\rm(\roman*)},leftmargin=2.2em]
		\item Taking $a=\E x$ in
		Theorem~\ref{thm:anchored-expectation} reduces to
		Lemma~\ref{lem:ricard-xu-expectation}.
		\item Taking $\N=\M$ and $\E=\id$ gives
		Corollary~\ref{cor:bregman-two-point}, and hence
		Theorem~\ref{thm:main1}.
		\item At $p=2$, one obtains the identity
		\begin{equation*}
				\|x\|_2^2
				=\|a\|_2^2+2\Ree\Tr\bigl(a^*(x-a)\bigr)
				+\|\E x-a\|_2^2+\|x-\E x\|_2^2.
		\end{equation*}
	\end{enumerate}
\end{cor}

\begin{proof}
	For \textup{(i)}, set $a=\E x$ in
	Theorem~\ref{thm:anchored-expectation}, use
	\eqref{eq:bregman-residual}, and multiply by $2$. For \textup{(ii)},
	set $\N=\M$ and $\E=\id$. For \textup{(iii)}, one has $J_2(a)=a$,
	and both quadratic estimates used above are Hilbert-space equalities.
\end{proof}

\section{Anchored martingale inequalities}
\label{sec:martingale}

In this section, we iterate the Bregman--Pythagorean decomposition from
Theorem~\ref{thm:bregman-pythagorean} along an increasing filtration of
von Neumann subalgebras. This yields an exact telescoping identity for
the Bregman divergence, from which we derive anchored martingale
inequalities for both finite and infinite filtrations. As a special
case, by choosing the anchor to be the initial conditional expectation,
we recover the martingale convexity inequality of Ricard and Xu~\cite[Corollary~3]{RX16}.

Let
$
	\M_0\subset\M_1\subset\cdots\subset\M_n\subset\M
$
be a finite increasing family of $\sigma^\varphi$-invariant von Neumann
subalgebras such that each restriction $\varphi|_{\M_k}$ is semifinite. Let
\[
\E_k:\M\longrightarrow\M_k
\]
be the corresponding $\varphi$-preserving conditional expectation. The tower
property reads
\begin{equation}\label{eq:tower}
	\E_j\E_k=\E_j,
	\qquad 0\le j\le k\le n.
\end{equation}

The exact decomposition from the preceding section iterates without loss.

\begin{prop}\label{prop:bregman-telescope}
	Let $a\in L_p(\M_0)$ and $x\in L_p(\M)$. Then
	\begin{equation}\label{eq:bregman-telescope}
		\mathcal D_p(x\mid a)
		=\mathcal D_p(x\mid\E_nx)
		+\sum_{k=1}^{n}\mathcal D_p(\E_kx\mid\E_{k-1}x)
		+\mathcal D_p(\E_0x\mid a).
	\end{equation}
\end{prop}

\begin{proof}
	Apply Theorem~\ref{thm:bregman-pythagorean} first to the expectation
	onto $\M_n$. Then iterate inside $\M_n$ through the expectations onto
	$\M_{n-1},\ldots,\M_1$. The identity $\E_{k-1}(\E_kx)=\E_{k-1}x$ follows from
	\eqref{eq:tower}, and the resulting terms are exactly those in
	\eqref{eq:bregman-telescope}.
\end{proof}

\begin{thm}\label{thm:finite-martingale}
	Let \(a\in L_p(\M_0)\) and \(x\in L_p(\M)\). If \(1<p\le2\), then
	\begin{equation}\label{eq:finite-martingale-lower}
			\mathcal D_p(x\mid a)
			\ge \frac{p-1}{2}\Bigg(
			\|\E_0x-a\|_p^2
			+\sum_{k=1}^{n}
			\|\E_kx-\E_{k-1}x\|_p^2
			+\|x-\E_nx\|_p^2
			\Bigg).
	\end{equation}
	If \(2\le p<\infty\), the inequality is reversed.
\end{thm}
\begin{proof}
	Suppose $1<p\le2$. By \eqref{eq:bregman-residual} and
	Lemma~\ref{lem:ricard-xu-expectation},
	\[
	\mathcal D_p(x\mid\E_nx)
	\ge \frac{p-1}{2}\|x-\E_nx\|_p^2.
	\]
	Applying \eqref{eq:bregman-residual} and the same lemma in
	$\M_k$ to the conditional expectation $\E_{k-1}|_{\M_k}$ gives
	\[
	\mathcal D_p(\E_kx\mid\E_{k-1}x)
	\ge \frac{p-1}{2}
	\|\E_kx-\E_{k-1}x\|_p^2,
	\qquad 1\le k\le n.
	\]
	Finally, Corollary~\ref{cor:bregman-two-point} in $L_p(\M_0)$ gives
	\[
	\mathcal D_p(\E_0x\mid a)
	\ge \frac{p-1}{2}\|\E_0x-a\|_p^2.
	\]
	Substitution into \eqref{eq:bregman-telescope} proves
	\eqref{eq:finite-martingale-lower}. For $2\le p<\infty$, all three
	estimates reverse direction.
\end{proof}

The infinite-filtration form follows from martingale mean convergence in
Haagerup \(L_p\)-spaces; see \cite[Remark~6.1(i)]{HJX10}, together with
the reduction framework of that paper for the present semifinite-weight
setting.
\begin{cor}\label{cor:infinite-martingale}
	Let $(\M_k)_{k\ge0}$ be an increasing sequence of
	$\sigma^\varphi$-invariant von Neumann subalgebras such that
	$\varphi|_{\M_k}$ is semifinite for every $k$, and let
	$\E_k:\M\to\M_k$ be the corresponding $\varphi$-preserving conditional
	expectations. Assume that $\bigcup_{k\ge0}\M_k$ is weak-star dense in
	$\M$. Let $a\in L_p(\M_0)$ and $x\in L_p(\M)$. If $1<p\le2$, then
	\begin{equation}\label{eq:infinite-martingale-lower}
		\mathcal D_p(x\mid a)
		\ge \frac{p-1}{2}
		\left(
		\|\E_0x-a\|_p^2
		+\sum_{k\ge1}\|\E_kx-\E_{k-1}x\|_p^2
		\right).
	\end{equation}
	If $2\le p<\infty$, the inequality is reversed. The series in
	\eqref{eq:infinite-martingale-lower} is understood in \([0,\infty]\).
\end{cor}

\begin{proof}
	Apply Theorem~\ref{thm:finite-martingale} to the first $n$ levels.
	Since $\E_nx\to x$ in $L_p(\M)$, the terminal residual tends to zero.
	For $1<p\le2$, the partial sums on the right-hand side are increasing,
	so passage to the limit gives \eqref{eq:infinite-martingale-lower}.
	For $2\le p<\infty$, the conclusion is immediate if the series
	diverges; if it is finite, pass to the limit in the finite-level upper
	estimate.
\end{proof}

Taking $a=\E_0x$ in Corollary~\ref{cor:infinite-martingale} and using
\eqref{eq:bregman-residual} gives
\begin{equation}\label{eq:rx-martingale-recovered}
	\|x\|_p^2
	\ge \|\E_0x\|_p^2
	+(p-1)\sum_{k\ge1}\|\E_kx-\E_{k-1}x\|_p^2,
	\qquad 1<p\le2,
\end{equation}
with the reverse inequality for $2\le p<\infty$. Inequality
\eqref{eq:rx-martingale-recovered} is the martingale convexity inequality
of Ricard and Xu~\cite[Corollary~3]{RX16}.

\section{Positive-cone tangent inequalities for \texorpdfstring{$0<p<1$}{0<p<1}}
\label{sec:quasi-banach}
This section concerns the quasi-Banach range
\(0<p<1\), where we prove
Proposition~\ref{prop:xu-tangent-no-direction} and
Theorems~\ref{thm:commutative-positive-cone-tangent}
and~\ref{thm:positive-cone-tangent}.

We first prove Proposition~\ref{prop:xu-tangent-no-direction}.

\begin{proof}[Proof of Proposition~\ref{prop:xu-tangent-no-direction}]
	Fix $0<p<1$ and work in $\ell_p^2$ with counting measure. For
	$x=(1,1)$ and $y=(1,1)$, the two sides of the formal extension of
	\eqref{eq:xu} are
	\[
	\|x+y\|_p^2=4\,2^{2/p}
	\quad\text{and}\quad
	\|x\|_p^2+2\langle j_p(x),y\rangle
	+(p-1)\|y\|_p^2=(p+2)2^{2/p}.
	\]
	Hence the left-hand side is strictly larger.
	
	For $x=(1,1)$ and $y=(1,-1)$, the corresponding values are
	\[
	\|x+y\|_p^2=4
	\quad\text{and}\quad
	\|x\|_p^2+2\langle j_p(x),y\rangle
	+(p-1)\|y\|_p^2=p\,2^{2/p}.
	\]
	The function $\phi(p)=p2^{2/p}$ satisfies
	\[
	\phi'(p)=2^{2/p}\left(1-\frac{2\log2}{p}\right)<0,
	\qquad 0<p\le1,
	\]
	and therefore $p2^{2/p}>4$ for $0<p<1$. Thus the opposite strict
	inequality also occurs.
\end{proof}

For \(0<p<1\), the squared quasi-norm is neither convex nor concave on the
whole space, and the Banach-space normalized duality map used above is not
available.
In the commutative argument below, the order condition \(x\pm y\ge0\)
implies \(x\ge |y|\), and hence \(y=0\) on \(\{x=0\}\). In the operator setting, \(A\pm B\ge0\) implies
\(\ker A\subseteq\ker B\), but it does not imply \(A\ge |B|\) in general. This distinction motivates the transfer argument below.

\subsection{The commutative positive cone}

Let $(\Omega,\mu)$ be a measure space, and let $L_p(\Omega;\mathbb R)$ be
equipped with
$
\|x\|_p=\left(\int_\Omega |x|^p\,d\mu\right)^{1/p}.
$

\begin{defin}\label{def:commutative-cone-tangent}
	Let $0<p<1$ and let $x,y\in L_p(\Omega;\mathbb R)$ satisfy
	$x\pm y\ge0$. Define
	\begin{equation}\label{eq:commutative-cone-tangent}
		\pair{J_{L_p,+}(x)}{y}
		:=
		\begin{cases}
			\displaystyle
			\|x\|_p^{2-p}\int_{\{x>0\}}x^{p-1}y\,d\mu,
			&x\ne0,\\[3mm]
			0,&x=0.
		\end{cases}
	\end{equation}
\end{defin}

The notation is directional: $J_{L_p,+}(x)$ need not belong to the continuous
dual of $L_p$. The integral in \eqref{eq:commutative-cone-tangent} is finite
because $|x^{p-1}y|\le x^p$. Moreover,
\begin{equation}\label{eq:commutative-cone-derivative}
	\pair{J_{L_p,+}(x)}{y}
	=\frac12\left.\frac{d}{dt}\right|_{t=0}\|x+ty\|_p^2.
\end{equation}
Indeed, the order assumption permits differentiation under the integral for
$|t|<1$, and the chain rule gives \eqref{eq:commutative-cone-derivative}.

The sharp quadratic constant follows from a scalar moment estimate.

\begin{lem}\label{lem:moment-p-below-one}
	Let \(0<p<1\), and let \(r\) be a real square-integrable random
	variable on a probability space. Then
\begin{equation}\label{eq:expectation-scalar}
		(2-p)(\mathbb E r)^2-(1-p)\mathbb E r^2
	\le
	p\,2^{2/p-2}\bigl(\mathbb E|r|^p\bigr)^{2/p},
\end{equation}
where the constant \(p\,2^{2/p-2}\) is optimal.
\end{lem}

\begin{proof}
	Set
$
	a=|\mathbb E r|,
	b=(\mathbb E r^2)^{1/2}
$ and $
	m=(\mathbb E|r|^p)^{1/p}.
$
	If \(a=0\), then the left-hand side of \eqref{eq:expectation-scalar} is nonpositive, and there is
	nothing to prove. By H\"older's inequality with conjugate exponents
	\(2-p\) and \((2-p)/(1-p)\),
	\[
	a
	\le \mathbb E|r|
	\le
	\bigl(\mathbb E|r|^p\bigr)^{1/(2-p)}
	\bigl(\mathbb E r^2\bigr)^{(1-p)/(2-p)}
	=
	m^{p/(2-p)}b^{2(1-p)/(2-p)}.
	\]
	Put \(s=b^2/a^2\ge1\). Rearranging the preceding estimate gives
	\[
	m^2\ge a^2s^{-2(1-p)/p}.
	\]
	If \((2-p)-(1-p)s\le0\), the desired inequality is immediate.
	Otherwise,
	\[
	\frac{(2-p)a^2-(1-p)b^2}{m^2}
	\le
	\bigl((2-p)-(1-p)s\bigr)s^{2(1-p)/p}
	=:g(s).
	\]
	A direct computation yields
	\[
	g'(s)
	=
	\frac{(1-p)(2-p)}{p}
	s^{2(1-p)/p-1}(2-s).
	\]
	Hence \(g\) attains its maximum at \(s=2\), and
$
	g(2)=p\,2^{2/p-2}.
$
	This proves the asserted inequality \eqref{eq:expectation-scalar}. Finally, equality is attained
	when \(r\) is the indicator of an event of probability \(1/2\);
	therefore, the constant \(p\,2^{2/p-2}\) is optimal.
\end{proof}

\begin{proof}[Proof of Theorem~\ref{thm:commutative-positive-cone-tangent}]
	Since \(x\pm y\ge0\) almost everywhere, we have
	\(x\ge |y|\ge0\) almost everywhere on \(\Omega\).
	In particular, \(y=0\) on \(\{x=0\}\). If \(x=0\), then \(y=0\),
	and the assertion is immediate. Assume henceforth that \(x\neq0\).
	
	For \(-1\le t\le1\), set
$
	z_t=x+ty,
$ and $
	F(t)=\|z_t\|_p^2
	=\left(\int_\Omega z_t^p\,d\mu\right)^{2/p}.
$
	For \(|t|<1\),
$
	z_t\ge (1-|t|)x.
$
	Let \(I\subset(-1,1)\) be compact and put
$
	\delta=1-\sup_{t\in I}|t|>0.
$
	Then \(z_t\ge\delta x\) for \(t\in I\). Since \(|y|\le x\), we have
	on \(\{x>0\}\)
	\[
	z_t^{p-1}|y|
	\le
	\delta^{p-1}x^p,
	\qquad
	z_t^{p-2}y^2
	\le
	\delta^{p-2}x^p.
	\]
	These bounds justify differentiation under the integral sign and
	show that \(F\in C^2(-1,1)\).
	
	Write
	\[
	A_t=\int_\Omega z_t^p\,d\mu,\qquad
	B_t=\int_{\{x>0\}}z_t^{p-1}y\,d\mu,\qquad
	C_t=\int_{\{x>0\}}z_t^{p-2}y^2\,d\mu.
	\]
	A direct differentiation gives
	\[
	\frac12F''(t)
	=
	(2-p)A_t^{2/p-2}B_t^2
	+
	(p-1)A_t^{2/p-1}C_t.
	\]
	Define a probability measure \(\nu_t\) by
$
	d\nu_t=\frac{z_t^p}{A_t}\,d\mu
$
	and set
$
	r_t=
	\begin{cases}
		y/z_t,&x>0,\\
		0,&x=0.
	\end{cases}
$
	Then
	\[
	B_t=A_t\mathbb E_{\nu_t}r_t,
	\qquad
	C_t=A_t\mathbb E_{\nu_t}r_t^2,
	\]
	and hence
	\[
	\frac12F''(t)
	=
	A_t^{2/p}
	\left[
	(2-p)(\mathbb E_{\nu_t}r_t)^2
	-
	(1-p)\mathbb E_{\nu_t}r_t^2
	\right].
	\]
	By Lemma~\ref{lem:moment-p-below-one},
	\[
	\frac12F''(t)
	\le
	p\,2^{2/p-2}
	A_t^{2/p}
	\bigl(\mathbb E_{\nu_t}|r_t|^p\bigr)^{2/p}.
	\]
	Moreover,
	\[
	A_t^{2/p}
	\bigl(\mathbb E_{\nu_t}|r_t|^p\bigr)^{2/p}
	=
	\left(\int_\Omega |y|^p\,d\mu\right)^{2/p}
	=
	\|y\|_p^2.
	\]
	Therefore,
	\[
	F''(t)
	\le
	2p\,2^{2/p-2}\|y\|_p^2,
	\qquad |t|<1.
	\]
	It follows that
$
	G(t)
	:=
	F(t)-p\,2^{2/p-2}t^2\|y\|_p^2
$
	is concave on \((-1,1)\). Since \(0\le z_t\le2x\) for
	\(-1\le t\le1\), dominated convergence shows that \(F\), and hence
	\(G\), is continuous on \([-1,1]\). The supporting-line inequality
	for \(G\) at \(0\) therefore gives
	\[
	G(1)\le G(0)+G'(0).
	\]
	By the definition of the positive-cone tangent form,
	\[
	F'(0)=2\pair{J_{L_p,+}(x)}{y}.
	\]
	Substituting this identity into the preceding inequality yields
	\[
	\|x+y\|_p^2
	\le
	\|x\|_p^2
	+2\pair{J_{L_p,+}(x)}{y}
	+p\,2^{2/p-2}\|y\|_p^2.
	\]
	
	It remains to prove optimality. Let
	\(\Omega=\{1,2\}\) with
	\(\mu(\{1\})=\mu(\{2\})=1/2\), and set
	\[
	x=(1,1),\qquad y_\varepsilon=(0,\varepsilon),
	\qquad 0<\varepsilon<1.
	\]
	Then \(x\pm y_\varepsilon\ge0\),
	\[
	\|x\|_p^2=1,
	\qquad
	2\pair{J_{L_p,+}(x)}{y_\varepsilon}=\varepsilon,
	\qquad
	\|y_\varepsilon\|_p^2
	=2^{-2/p}\varepsilon^2.
	\]
	Furthermore,
	\[
		\begin{aligned}
			\|x+y_\varepsilon\|_p^2
			&=
			\left(
			\frac{1+(1+\varepsilon)^p}{2}
			\right)^{2/p}\\
			&=1+\varepsilon+\frac p4\varepsilon^2
			+o(\varepsilon^2).
		\end{aligned}
	\]
	If the coefficient \(p\,2^{2/p-2}\) were replaced by \(C\), then
	the preceding expansion would imply
	\[
	\frac p4\le C\,2^{-2/p}.
	\]
	Thus \(C\ge p\,2^{2/p-2}\), proving that the stated coefficient is
	optimal. Replacing \(\varepsilon\) by \(-\varepsilon\) in the same
	expansion and adding gives
	\[
	\|x+y_\varepsilon\|_p^2+
	\|x-y_\varepsilon\|_p^2
	=2+\frac p2\varepsilon^2+o(\varepsilon^2).
	\]
	Since \(\|y_\varepsilon\|_p^2=2^{-2/p}\varepsilon^2\), this also
	proves the optimality of the coefficient \(p\,2^{2/p-1}\).
\end{proof}

\subsection{Transfer to Schatten \texorpdfstring{$p$}{p}-classes}

The transfer rests on the two-dimensional $L_p$ representation furnished by
tracial joint spectral measures.

\begin{prop}\label{prop:positive-schatten-plane}
	Let $0<p<\infty$ and let 
	$C,D\in S_p$ be positive. Then there exist a measure space
	$(\Omega,\mu)$ and nonnegative functions $c,d\in L_p(\Omega;\mathbb R)$
	such that
	\begin{equation}\label{eq:positive-schatten-plane}
		\|\alpha C+\beta D\|_{S_p}
		=\|\alpha c+\beta d\|_{L_p(\mu)},
		\qquad \alpha,\beta\in\mathbb R.
	\end{equation}
\end{prop}

\begin{proof}
	Assume first that $C$ and $D$ are matrices. By Hein\"avaara's tracial
	joint spectral-measure representation~\cite[Theorem~1.3]{Hei25}, there
	is a positive Borel measure $\mu$ on $\mathbb R^2$ such that
	\[
	\|\alpha C+\beta D\|_{S_p}^p
	=p(p+1)\int_{\mathbb R^2}|\alpha s+\beta t|^p\,d\mu(s,t).
	\]
By the positive-coordinate support property
\cite[Lemma~1.5]{Hei25}, the positivity of \(C\) implies that \(\mu\) is
supported in \(\{(s,t):s\ge0\}\). The proof of the same lemma, with the
two coordinates interchanged, shows that \(D\ge0\) forces support in
\(\{(s,t):t\ge0\}\). Hence
\(\supp\mu\subset\mathbb R_+^2\). Absorbing the factor \(p(p+1)\) into
the measure gives
	\eqref{eq:positive-schatten-plane} in finite dimension.
	
	For general positive compact operators, choose positive finite-rank
	$C_m,D_m$ converging to $C,D$ in $S_p$. View each pair in the
	finite-dimensional matrix algebra supported on the sum of its ranges,
	and let $\mu_m$ be the corresponding measures on $\mathbb R_+^2$. Set
	\[
	K=\{(u,v)\in\mathbb R_+^2:u+v=1\}.
	\]
	Define a finite measure $\nu_m$ on $K$ by
	\[
	\int_K f\,d\nu_m
	=\int_{\mathbb R_+^2\setminus\{0\}}
	(s+t)^p
	f\left(\frac{s}{s+t},\frac{t}{s+t}\right)d\mu_m(s,t).
	\]
	Since \(K\) is compact metrizable, \(C(K)\) is separable. Moreover,
	\(\nu_m(K)=\|C_m+D_m\|_{S_p}^p\) is uniformly bounded. Hence, after
	passing to a subsequence,
	\(\nu_m\overset{*}{\rightharpoonup}\nu\) for a finite measure \(\nu\)
	on \(K\). For every \(\alpha,\beta\in\mathbb R\),
	\[
	\begin{aligned}
		\|\alpha C+\beta D\|_{S_p}^p
		&=\lim_{m\to\infty}\|\alpha C_m+\beta D_m\|_{S_p}^p\\
		&=\lim_{m\to\infty}\int_K|\alpha u+\beta v|^p\,d\nu_m(u,v)\\
		&=\int_K|\alpha u+\beta v|^p\,d\nu(u,v).
	\end{aligned}
	\]
	Here the first limit follows from norm continuity when \(p\ge1\).
	When \(0<p<1\), it follows from
	\[
	\bigl|\|X_m\|_{S_p}^p-\|X\|_{S_p}^p\bigr|
	\le \|X_m-X\|_{S_p}^p,
	\]
	which is a consequence of the \(p\)-triangle inequality. Since
	$0\le u,v\le1$ on $K$ and $\nu(K)<\infty$, the functions
	$c(u,v)=u$ and $d(u,v)=v$ belong to $L_p(K,\nu)$. Taking
	$(\Omega,\mu)=(K,\nu)$ completes the proof.
\end{proof}

Next, we prove Theorem~\ref{thm:positive-cone-tangent}.

\begin{proof}[Proof of Theorem~\ref{thm:positive-cone-tangent}]
	Let $A,B\in S_p$ be self-adjoint and assume $A\pm B\ge0$. Put
$
	C=A+B$ and $ D=A-B.
$
	Apply Proposition~\ref{prop:positive-schatten-plane} to $C,D$, and set
$
	x=(c+d)/2$ and $y=(c-d)/2.
$
	Then $x\pm y\ge0$, and
	\begin{equation}\label{eq:same-line-function}
		\|A+tB\|_{S_p}=\|x+ty\|_{L_p(\mu)},
		\qquad t\in\mathbb R.
	\end{equation}
	
	Choose $x,y$ as in \eqref{eq:same-line-function}. The same identity at
	$t=0,\pm1$ gives
	\[
	\|A\|_{S_p}=\|x\|_{L_p},
	\qquad
	\|A\pm B\|_{S_p}=\|x\pm y\|_{L_p},
	\]
	while Proposition~\ref{prop:positive-schatten-plane} gives
	$\|B\|_{S_p}=\|y\|_{L_p}$. Differentiating the squared form of
	\eqref{eq:same-line-function} at $t=0$ identifies the two cone tangent
	terms. Theorem~\ref{thm:commutative-positive-cone-tangent} now 
	gives Theorem~\ref{thm:positive-cone-tangent}.

	For the optimality statements, it suffices to consider diagonal
	\(2\times2\) matrices. Let
	\[
	A=\begin{pmatrix}1&0\\0&1\end{pmatrix},
	\qquad
	B_\varepsilon=\begin{pmatrix}0&0\\0&\varepsilon\end{pmatrix},
	\qquad 0<\varepsilon<1.
	\]
	Then $A\pm B_\varepsilon\ge0$,
	$\|B_\varepsilon\|_{S_p}^2=\varepsilon^2$, and
	\begin{equation}\label{eq:positive-cone-sharp-expansion}
		\begin{aligned}
			\|A+B_\varepsilon\|_{S_p}^2
			&=\bigl(1+(1+\varepsilon)^p\bigr)^{2/p}=2^{2/p}+2^{2/p}\varepsilon
			+p2^{2/p-2}\varepsilon^2+o(\varepsilon^2).
		\end{aligned}
	\end{equation}
	On the other hand,
	\[
	\|A\|_{S_p}^2=2^{2/p},
	\qquad
	2\pair{J_{S_p,+}(A)}{B_\varepsilon}=2^{2/p}\varepsilon.
	\]
	Thus the quotient of the second-order remainder by
	$\|B_\varepsilon\|_{S_p}^2$ tends to
	$p\,2^{2/p-2}$, proving sharpness in
	Theorem~\ref{thm:positive-cone-tangent}. 
\end{proof}
Replacing \(\varepsilon\) by
\(-\varepsilon\) in \eqref{eq:positive-cone-sharp-expansion} and adding
gives
\[
\|A+B_\varepsilon\|_{S_p}^2+
\|A-B_\varepsilon\|_{S_p}^2
=2^{2/p+1}+p\,2^{2/p-1}\varepsilon^2
+o(\varepsilon^2).
\]
Since \(\|B_\varepsilon\|_{S_p}^2=\varepsilon^2\), this proves the
optimality asserted in Corollary~\ref{cor:positive-cone-midpoint}.

%\section*{Declaration of competing interest}
%The authors declare no competing interests.
%
%\section*{Data availability}
%No data were used for the research described in this article.

\end{document}